\documentclass[a4paper]{amsart}

\usepackage{graphicx}
\usepackage{amsmath}
\usepackage{amscd}
\usepackage{amsfonts}
\newtheorem{theorem}{Theorem}
\newtheorem{corollary}{Corollary}
\newtheorem{lemma}[theorem]{Lemma}
\theoremstyle{definition}
\newtheorem{definition}[theorem]{Definition}

\theoremstyle{remark}

\newcommand{\NN}{\ensuremath{\mathbb{N}}}
\newcommand{\CC}{\ensuremath{\mathbb{C}}}
\newcommand{\RR}{\ensuremath{\mathbb{R}}}

\newcommand{\norm}[2][]{\| \, {#2} \,\|_{#1}}
\newcommand{\abs}[1]{\lvert{#1}\rvert}

\newcommand{\card}[1]{\mbox{card}({#1})}
\newcommand{\la}{l}
\newcommand{\e}{e}

\title{Relatively compact sets in the reduced $C^{\ast}$-algebras of Coxeter groups}
\author{Gero Fendler}
\address{NuHAG\\
Faculty of Mathematics\\
University of Vienna\\
Nordbergstasse 15\\
1090 Vienna\\
Austria}
\email{gero.fendler<at>univie.ac.at}

\begin{document}

\begin{abstract}
We characterize relatively norm compact sets in the 
regular $C^*$-algebra of finitely generated Coxeter groups
using a geometrically defined positive semigroup acting on the algebra.
\end{abstract}

\maketitle

\section{Introduction}

Let $(X,d)$ be a compact metric space, $x_0\in X$.
In $C(X)$, the continuous complex valued functions on X, 
consider the convex, balanced and closed set

\[\mathcal{K}=\{ f : \abs{f(x)-f(y)}\leq d(x,y), f(x_0)=0\}.\]

The Arzela-Ascoli theorem shows that 
$\mathcal{K}$ is relatively compact.
On the other hand this theorem can be thought to compare
any relatively compact set against this special set.
\par
 In the non-commutative context this has been made precise 
by Antonescu and Christensen~\cite{AtChr} as follows: 

Let $A$ be a unital, separable  $C^*$-algebra and $\mathcal{S}$ the set of its states endowed with the $w^*$-topology.
\begin{definition}
$\mathcal{K}\subset A$ is called a metric set if
it is convex, balanced {\em norm compact } and {\em separates the states} of $A$
\end{definition}

\begin{lemma}[\cite{AtChr}]
If $\mathcal{K}\subset A$ is a metric set then
\[d(\varphi,\psi):= \sup_{x\in\mathcal{K}}\abs{\varphi(x)-\psi(x)},\quad \varphi,\psi\in\mathcal{S}\]
defines a metric on $\mathcal{S}$, which generates the $w^*$-topology.
\end{lemma}
Their {\em general non-commutative Arzela-Ascoli Theorem}
reads as follows
\begin{theorem}[\cite{AtChr}]
Let $A$ be a unital $C^*$-algebra  $\mathcal{K}\subset A$  a metric set then
$\mathcal{H}\subset A$ is relatively compact if and only if
$\mathcal{H}$ is bounded and for all $\epsilon>0$ exists $N>0$ such that
\[\mathcal{H}\subset B_{\epsilon}(0) + N\mathcal{K}+\CC\mbox{Id},\]
where $B_{\epsilon}(0)\subset A$ is the ball of radius $\epsilon$ around $0$.
\end{theorem}
Our aim here is to 
give an example of some such set $\mathcal{K}$ in the reduced 
$C^{\ast}$-algebra $A=C^{\ast}_{\lambda}(G)$ of a finitely generated  Coxeter group
$G$.
\par
Let $G,S$ be a Coxeter group and $l$ the length function associated to the generating set $S$. (For the readers convenience in the next two  sections we recall some 
notions and 
assertions related to the regular $C^*$-algebra of Coxeter groups.)
\begin{theorem}\label{thm:main}
\[\mathcal{K}=\left\{\lambda(f): \norm{\lambda(f)}\leq 1 
\mbox{ and }
\norm{\lambda( \la \cdot f)} \leq 1 \right\}\]
is relatively compact in $C^{\ast}_{\lambda}(G)$.
\end{theorem}
The proof of this theorem is given in our last section.
\par
Since the set $\mathcal{K}$ in $C^{\ast}_{\lambda}(G)$ separates the states,
is convex and balanced an application
of the theorem of Antonescu and Christensen  characterizes
relatively compact subsets of $C^{\ast}_{\lambda}(G)$ as follows:
\begin{corollary}
A set $\mathcal{H}\subset C^{\ast}_{\lambda}(G)$ is relatively compact
if and only if it is bounded and for all $\epsilon>0$ there is $m\in\NN$
such that
\[ \mathcal{H}\subset m\mathcal{K}+ \CC\lambda(\delta(e))+B_{\epsilon}(0),\]
where $B_{\epsilon}(0)\subset C^{\ast}_{\lambda}(G)$ is the ball of radius 
$\epsilon$ and center $0$.
\end{corollary}
\section{Coxeter group}
\begin{definition}
A pair $(G,S)$ is a Coxeter group if,
$S$  is a finite generating subset of the group $G$ with the following {{presentation}}:
\begin{eqnarray*}
{s^2}&{=}&{\e} \quad {s\in S,}\\
{(st)^{m(s,t)}}&{=}&{\e} \quad {s,t \in S, \ s\neq t,}
\end{eqnarray*}
where ${m(s,t) \in \{2,3,4, \ldots, \infty\}}$.
\end{definition}

A specific tool for working with Coxeter groups is their
{\em geometric representation}.

Let
${V}{= \oplus_{s\in S} \RR \alpha_s}$ be an abstract real vector space with basis
${\{\alpha_s : s \in S\}}$.
Define a bilinear form on it:
${B(\alpha_s,\alpha_t)= %
\left\{\begin{array}{ll}
1&s=t\\
- \cos \frac{\pi}{m(s,t)}&m(s,t)\neq \infty\\%
-1&m(s,t) = \infty
\end{array} \right.}$.
For ${s\in S}$ define a reflection by
${\sigma_s \xi} {= \xi - 2 B(\alpha_s,\xi) \alpha_s}$.
Then
\begin{description}
\item[{$\circ$}]
${V= \RR\alpha_s \oplus H_s}$, where\\ 
${H_s=\{\xi:B(\alpha_s,\xi)=0\}}$ is stabilized point wise
by ${ \sigma_s}$
and ${ \sigma_s \alpha_s = - \alpha_s}$.
\item[{$\circ$}]
${s \mapsto \sigma_s}$ extends multiplicatively to a representation
${\sigma : G \rightarrow {\rm Gl}(V)}$
of the Coxeter group.
\item[{$\circ$}] ${\sigma}$ is faithful and ${\sigma(G)}$ a 
discrete subgroup of ${{\rm Gl}(V)}$.
\end{description}

{We dualise the representation $\sigma$ to obtain the adjoint representation}
\[{\sigma^{\ast}(g)f \ (\xi)= f ( \sigma(g^{-1}\xi)),\ %
f\in V^{\ast},  \xi \in V}\]

For $s\in S$ let $Z_s$ be the hyperplane
${Z_s=\{f\in V^{\ast}: f(\alpha_s)=0\}}$, and
$A_s$ the halfspace
${A_s=\{f\in V^{\ast}: f(\alpha_s)>0\}}$; 
define a family of Hyperplanes in ${V^{\ast}}$
${\mathcal{H}= \cup_{g\in G}\ gZ_s}$.
Denote ${C= \cap_{s\in S}\ A_s}$ the intersection of the halfspaces,
its closure ${D= \overline{C}\setminus \{0\}}$ is called the 
{\em fundamental chamber} usually considered as a subset of the union of its 
translates ${U=\cup_{g\in G}\ gD}$, the {\em Tit's cone}. 

The following facts hold true:
\begin{description}
\topsep0cm
\partopsep0cm
\itemsep0cm
\item[(i)]%
${C}$ is a simplicial cone, its faces are the sets ${Z_s\cap D}$.
\item[(ii)]%
${U}$ is a convex cone, ${D}$ a fundamental domain for the action of
${G}$ on it.
\item[(iii)]%
a closed line segment ${[u,c] \subset U}$ meets only finitely many members
of ${\mathcal{H}}$.
\item[(iv]%
Moreover, for any ${c\in C}$:\\
${\card{ \{ Z \in \mathcal{H} : [gc,c]\cap Z \neq \emptyset\}} = \la(g)}$,
\end{description}

where
 ${\la(g) = \inf \{k: g= s_1 \ldots s_k, \ s_i \in S\}}$ denotes the
usual length with respect to the generating set ${S}$.
This construction due to Tits was used by 
Bo{\.z}ejko, M. and Januszkiewicz, T. and Spatzier, R. J., 
we recall their short proof of their theorem
\begin{theorem}[\cite{BJS}]\label{thm:length}
For $t>0$ 
\[\varphi_t:g\mapsto e^{-t \la (g)}\]
is a positive definite function on $G$. 
\end{theorem}

\begin{proof}
\begin{eqnarray*}{\la(g^{-1}h))}&=&{\card{\{ Z \in \mathcal{H} : [hc,gc]\cap Z
  \neq \emptyset\}}}\\
&=&{\sum_{Z\in \mathcal{H}} \abs{\chi_{h}(Z) - \chi_{g}(Z)}^2},
\end{eqnarray*}
where ${c\in C}$ is arbitrary and ${\chi_h}$ is the characteristic
function of
${N^h = \{ Z \in \mathcal{H} : [hc,c]\cap Z \neq \emptyset\}}$.\\
Hence ${\la(.)}$ is negative definite and, by a theorem of 
Schoenberg~\cite{Schoenberg38} 
(we only need the part already due to  Schur~\cite{Schur11}), 
${e^{-t\la(.)}}$ is positive definite, see e.g.~\cite[Theorem 7.8]{BF75}. 
\end{proof}

\section{Regular representation}

For functions $f,h: G \to \CC$ their convolution is defined by:
\[
f\ast h (y) = \sum_{x\in G} f(x)h(x^{-1}y).\]
For summable $f:G\to\CC$ we denote $\lambda(f):l^2(G)\to l^2(G)$ the
associated
convolution operator $\lambda(f)h= f\ast h$.
The regular (or reduced) $C^{\ast}$--algebra $C^{\ast}_{\lambda}(G)$ is the just the operator norm
closure of $\{\lambda(f):\,f\in l^1(G)\}$.
Denote for $g\in G$ $\delta_g$ the point mass one in $g\in G$ then
$\lambda(\delta_g)$ is just left translation by $g^{-1}$ on $l^2(G)$
and we are just dealing with the integrated version of the left regular
representation. Since for $A\in C^{\ast}_{\lambda}(G)$ there is a unique
$f=A\delta_e\in l^2$ we abuse notation to denote $A=\lambda(f)$.

The Tits cone with its devision by the hyperplanes can be seen as a 
subset of a cubical building. 
This allows to estimate certain convolution operator norms.
The first example of such an estimation was given for the free group on two generators by 
U.~Haagerup~\cite{Haa79} and accordingly such inequalities are called 
Haagerup inequality. Versions more appropriate for our purpose appear in
\cite{RRS,F03,Cha,Tal}:

\begin{theorem}\label{thm:haaiq}
A Coxeter group is a group of rapid decay:\\
there is $C>0$ and $k\in\NN$ such that
\[\norm{\lambda(f)}\leq C\left(\sum_g\abs{f(g)}^2(1+l(g))^{2k}\right)^{\frac{1}{2}}.\]
\end{theorem}
A consequence of this theorem is the following lemmata due to Haagerup~\cite{Haa79,Haa81}. For the readers convenience we recall their proofs.
\begin{lemma}
If $\varphi:G\to G$ is such that
$\sup_g \abs{\varphi(g)}(1 + l(g))^k < \infty$, then for all $\lambda(f)\in C^{\ast}_{\lambda}(G)$
\[\norm{\lambda(\varphi\cdot f)}\leq C \sup_g \abs{\varphi(g)}(1 + l(g))^k\norm{\lambda(f)}.\]
Here $C$ and $k$ are the constants in the Haagerup inequality. 
\end{lemma}
\begin{proof}
From $\lambda(f)\delta_e=f$ we have 
$\sum_g\abs{f(g)}^2)^{\frac{1}{2}}=\norm[2]{f} \leq\norm{\lambda(f)}$.
and by the Haagerup inequality:
\begin{eqnarray*}
\norm{\lambda(\varphi\cdot f)}^2&\leq&
C \sum_g\abs{\varphi(g)f(g)}^2(1+l(g))^{2k}\\
&\leq&\sum_g\abs{f(g)}^2C \sup_g \abs{\varphi(g)}^2(1 + l(g))^{2k}
\end{eqnarray*}
\end{proof}
\begin{lemma}
There is a sequence of finitely supported functions $(\psi_m)$ 

such that
for $\lambda(f)\in C^{\ast}_{\lambda}(G)$:

\begin{itemize}
\item $\lambda(\psi_m\cdot f) \to \lambda(f), \mbox{ as } n\to\infty$

\item $\norm{\lambda(\psi_m\cdot f)}\leq 3\norm{\lambda(f)}$ 
\end{itemize}
\end{lemma}

\begin{proof}
Since, by theorem~\ref{thm:length},
the functions $\varphi_t$ are positive definite, they define contractive (i.e. norm non-increasing)
multiüpliers on the regular $C^*$-algebra.
Let
\[\varphi_{n,t}=\left\{\begin{array}{ll}
e^{-tl(g)}& \mbox{ if }l(g)\leq n\\
0&\mbox{ else }
\end{array}\right.\]
Then 
\begin{eqnarray*}
\norm{\lambda(\varphi_{n,t}\cdot f)-\lambda(f)}&\leq&
\norm{\lambda(\varphi_{n,t}\cdot f)-\lambda(\varphi_{t} \cdot f)}+\\
&&\norm{\lambda(\varphi_{t}\cdot f)-\lambda(f)}\\
&\leq& C \sup_{l>n}{e^{-tl}(1+l)^k}\norm{\lambda(f)}+\\
&&
\norm{\lambda(\varphi_{t}\cdot f)-\lambda(f)}
\end{eqnarray*}

Since $\sup_{l>n}{e^{-tl}(1+l)^k}\to 0$ as $n\to\infty$
we can extract the $\psi_m$ from the $\varphi_{n,t}$.
\end{proof}

\section{Relatively compact sets}
First we notice that the  positive definite functions
$\varphi_t:g\mapsto e^{-tl(g)}$ define a
$C_0$-semigroup of multipliers on $C^{\ast}_{\lambda}(G)$
given by
$M_t:C^{\ast}_{\lambda}(G)\to C^{\ast}_{\lambda}(G)$, 
$\lambda(f)\mapsto \lambda(\varphi_t \cdot f)$.

\begin{lemma}
$M:t\mapsto M_t$ is a $C_0$-semigroup of
contractions on $C^{\ast}_{\lambda}(G)$.
\end{lemma} 
\begin{proof}
Since $\varphi_t$ is positive definite
\[\norm{M_t}=\varphi_t(e)=1.\]
For finitely supported $f$ everything is elementary and now an approximation proves the assertion.
\end{proof}

\begin{lemma}
The generator $D$ of the semigroup $M_t$
is given by
\[D(\lambda(f))= - \lambda(l\cdot f)\]
\[\mbox{Dom}(D)=
\left\{\lambda(f):\lambda(l\cdot f)\in C^{\ast}_{\lambda}(G)\right\}\]
\end{lemma}

\begin{proof}
We have
\[\lambda(\varphi_{t}\cdot \delta_{g})=e^{-t \la (g)}\lambda(\delta_{g}),\]
hence the assertion is clear for finitely supported $f=\sum_g f(g)\delta_g$.

Now as a generator of a $C_0$-contraction semi group the operator $D$ 
has a closed graph.
But if
\[\lambda(f) \mbox{ and } \lambda(\la\cdot f) \in C^{\ast}_{\lambda}(G),\]
then for the finitely supported $\psi_m$ as above:
\[\lambda(\psi_m\cdot f)\to \lambda(f)\]
and
\[\lambda(\psi_m\cdot\la \cdot f)\to \lambda(\la \cdot f).\]
\end{proof}

\begin{proof}[Proof of Theorem~\ref{thm:main}]
We shall show that for $\epsilon>0$ there exists a finite dimensional
 bounded set
\[\tilde{\mathcal{K}}_{\epsilon}\subset C^{\ast}_{\lambda}(G)\]
such that for all $f\in  \mathcal{K}$
\[ \mbox{dist}(f,\mathcal{\tilde{K}}_{\epsilon})\leq \epsilon.\]
(this show that $\mathcal{K}$ is totally bounded)

We have for $f\in\mathcal{K}$:
\[\lambda(\varphi_t\cdot f)-\lambda(f)=M_t(\lambda(f))-\lambda(f)=
\int_0^t D(M_s(\lambda(f))\,ds\]
Hence
\begin{eqnarray*}
\norm{\lambda(\varphi_t\cdot f)-\lambda(f)}&\leq&
t \sup_{s<t}\norm{\lambda(\la e^{-s\la}\cdot f)}\\
&\leq&
t \norm{\lambda(\la\cdot f)}\;\leq \;t.
\end{eqnarray*}

and
\begin{eqnarray*}
\norm{\lambda(\varphi_t\cdot f)-\lambda(\varphi_{n,t}\cdot f)}&\leq&
C \sup_{l>n}e^{-tl}(1+l)^k\norm{\lambda(f)}\\
&\leq&C \sup_{l>n}e^{-tl}(1+l)^k
\end{eqnarray*}

taking first $t$ small and then $n$ large we have an approximation to
$\lambda(f)$ by certain $\lambda(\varphi_{n,t}\cdot f)$ up to $\epsilon$
uniformly in $\lambda(f)\in \mathcal{K}$. Further for this $n$
\begin{eqnarray*}
\norm{\lambda(\varphi_{n,t}\cdot f)}&\leq&
\norm{\lambda(\varphi_t\cdot f)-\lambda(\varphi_{n,t}\cdot f)}+
\norm{\lambda(\varphi_t\cdot f)}\\
&\leq&C (\sup_{l>n}e^{-tl}(1+l)^k+1)\norm{\lambda(f)}\\
&\leq&C (\sup_{l>n}e^{-tl}(1+l)^k+1)
\end{eqnarray*}
So these $\lambda(\varphi_{n,t}\cdot f)$ are from a bounded 
set and all have their support in words of length at most $n$.
The functions with support in this finite set give rise to a finite 
dimensional subspace of $C^{\ast}_{\lambda}(G)$.
\end{proof}

\bibliographystyle{abbrv}

\end{document}